\documentclass[a4paper,11pt]{article}
\usepackage{mathrsfs, amsmath, amsxtra, amssymb, latexsym, amscd, amsthm}
\def\E{\mathbb{E}}

\newtheorem{thm}{Theorem}[section]

\newtheorem{lem}{Lemma}[section]

\theoremstyle{definition}

\theoremstyle{remark}
\newtheorem{rem}{Remark}[section]
\numberwithin{equation}{section}

\begin{document}

\title{Some identities for the Narayana polynomial and its derivatives}
\author{Nguyen Tien Dung\thanks{Department of Mathematics, VNU University of Science, Vietnam National University, Hanoi, 334 Nguyen
Trai, Thanh Xuan, Hanoi, 084 Vietnam. Email: dung@hus.edu.vn}\,\,\,\footnote{Department of Mathematics, FPT University, Hoa Lac High Tech Park, Hanoi, Vietnam.}}

\date{June 18, 2022}          
\maketitle
\begin{abstract}
In this note, by employing a nice property of semicircular distributions, we derive  some identities for the Narayana polynomial and its derivatives.
\end{abstract}
\section{Introduction}
The Narayana numbers $N(r,k)$ and the Narayana polynomials $\mathcal{N}_{r}(z)$ are defined by
\begin{equation}
N(r,k) = \frac{1}{r} \binom{r}{k-1} \binom{r}{k},
\label{nara-numb-def}
\end{equation}
\begin{equation}
\mathcal{N}_{r}(z) = \sum_{k=1}^{r} N(r,k) z^{k-1}.
\label{nara-poly-def}
\end{equation}
Because of their applications in Combinatorics theory, the Narayana numbers and the Narayana polynomials have been investigated by several authors. There is a large number of combinatorial properties can be found in the literature, see \cite{Amdeberhan2013,Chen2010,Lassalle2012,Mansour2009} and references therein.

The main result of this paper can be formulated as follows.
\begin{thm}\label{hj}For every $r\geq 3,$ we have
\begin{equation}\label{ikldd}
\mathcal{N}_r(z)=\frac{2r-1}{r+1}(1+z)\mathcal{N}_{r-1}(z)-\frac{r-2}{r+1}(1-z)^2\mathcal{N}_{r-2}(z),
\end{equation}
\begin{equation}\label{ikld}
\mathcal{N}'_r(z)=\frac{r-1}{2z}\mathcal{N}_r(z)+\frac{(r-1)(z-1)}{2z}\mathcal{N}_{r-1}(z),
\end{equation}
\begin{equation}\label{iklda}
\mathcal{N}''_r(z)=-\frac{r-1}{z^2}\mathcal{N}_r(z)+\left(\frac{(r-1)^2}{z}+\frac{r-1}{z^2}\right)\mathcal{N}_{r-1}(z).
\end{equation}
\end{thm}
The proof of the above theorem is deferred to Section \ref{gh3}. Let us end up this Section with some remarks.
\begin{rem}In his paper \cite{Lassalle2012}, Lasalle established the following recurrence
\begin{equation}
(z+1)\mathcal{N}_{r}(z) - \mathcal{N}_{r+1}(z)
= \sum_{n \geq 1} (-z)^{n} \binom{r-1}{2n-1}
A_{n} \mathcal{N}_{r-2n+1}(z),
\label{def0-A}
\end{equation}
where the numbers $A_{n}$ satisfies the recurrence
\begin{equation*}
(-1)^{n-1}A_{n} = C_{n} +
\sum_{j=1}^{n-1} (-1)^{j} \binom{2n-1}{2j-1} A_{j}C_{n-j},
\end{equation*}
\noindent
with $A_{1} = 1$ and $C_{n} = \frac{1}{n+1} \binom{2n}{n}$ the Catalan number. Interestingly, our relation (\ref{ikldd}) is simpler than (\ref{def0-A}).
\end{rem}
\begin{rem}From (\ref{ikldd}) and (\ref{ikld}) we obtain the following recurrence for the first order derivative
$$\mathcal{N}'_r(z)=\frac{r-1}{2r+2}\frac{r-2+3rz}{z}\mathcal{N}_{r-1}(z)-\frac{(r-1)(r-2)}{2r+2}\frac{(1-z)^2}{z}\mathcal{N}_{r-2}(z).$$
Similarly, we can express $\mathcal{N}''_r(z)$ in term of $\mathcal{N}_{r-1}(z)$ and $\mathcal{N}_{r-2}(z).$
\end{rem}
\begin{rem} It is known that $C_r:=\mathcal{N}_r(1),r\geq 1$ is the sequence of the Catalan numbers. Hence, by choosing $z=1,$ we obtain from (\ref{ikldd}), (\ref{ikld}) and (\ref{iklda}) the following identities
$$C_r=\frac{4r-2}{r+1}C_{r-1},$$
$$\sum_{k=1}^{r} (k-1)N(r,k)=\frac{r-1}{2}C_r,$$
$$\sum_{k=1}^{r} (k-1)(k-2)N(r,k)=-(r-1)C_r+r(r-1)C_{r-1}.$$
Similarly, by choosing $z=2,$ we also obtain the identities for the sequence of the large Schr\"oder numbers $S_r:=\frac{1}{2}\mathcal{N}_r(2),r\geq 1.$
\end{rem}
\section{Proofs}\label{gh3}
Let $X$ be a  random variable with the semicircular distribution, i.e. $X$ has a density function $f(x)$ given by
\begin{equation}
f(x) = \begin{cases}
        \frac{2}{\pi} \sqrt{1-x^{2}} & \quad \text{ if } -1 \leq x \leq 1 \\
       0  & \quad \text{ otherwise}.
     \end{cases}
\end{equation}
It is known from \cite{Amdeberhan2013} that the Narayana polynomial can be expressed as follows
\begin{equation}\label{nana}
\mathcal{N}_r(z)=\E[(1+z+2\sqrt{z}X)^{r-1}],
\end{equation}
where $\E$ denotes the expectation operator.

In order to prove Theorem \ref{hj}, we need to use the following property of semicircular distributions.
\begin{lem}\label{opk} For any continuously differentiable function $h,$ we have
$$\E[h(X)X]=\frac{1}{3}\E[h'(X)(1-X^2)].$$
\end{lem}
\begin{proof}We observe that
$$F(x):=\int_x^1yf(y)dy=\frac{2}{\pi}\int_x^1y\sqrt{1-y^{2}}dy=\frac{2}{3\pi}(1-x^2)\sqrt{1-x^{2}},\,\,-1\leq x\leq 1.$$
Hence, by integrating by parts, we obtain
\begin{align*}
\E[h(X)X]&=\int_{-1}^1h(x)xf(x)dx=-\int_{-1}^1h(x)dF(x)\\
&=\int_{-1}^1F(x)h'(x)dx=\int_{-1}^1\frac{1}{3}(1-x^2)h'(x)f(x)dx\\
&=\frac{1}{3}\E[h'(X)(1-X^2)].
\end{align*}
The proof of the lemma is complete.
\end{proof}

We now are ready to present the proof of Theorem \ref{hj}.

\noindent {\it Proof of (\ref{ikldd}).}  Using the representation (\ref{nana}), we have
\begin{align*}
\mathcal{N}_r(z)&=\E[(1+z+2\sqrt{z}X)^{r-2}(1+z+2\sqrt{z}X)]\\
&=(1+z)\mathcal{N}_{r-1}(z)+2\sqrt{z}\E[(1+z+2\sqrt{z}X)^{r-2}X],
\end{align*}
and hence,
\begin{equation}\label{io}
\mathcal{N}_r(z)-(1+z)\mathcal{N}_{r-1}(z)=2\sqrt{z}\E[(1+z+2\sqrt{z}X)^{r-2}X].
\end{equation}
Furthermore, by using Lemma \ref{opk} and the relation (\ref{io}), we deduce
\begin{align*}
\mathcal{N}_r(z)-(1+z)\mathcal{N}_{r-1}(z)
&=\frac{4(r-2)}{3}z\E[(1+z+2\sqrt{z}X)^{r-3}(1-X^2)].
\end{align*}
This gives us the following
\begin{equation}\label{ioa}
z\E[(1+z+2\sqrt{z}X)^{r-3}X^2]=-\frac{3}{4(r-2)}\left(\mathcal{N}_r(z)-(1+z)\mathcal{N}_{r-1}(z)\right)+z\mathcal{N}_{r-2}(z).
\end{equation}
Once again we use the representation (\ref{nana}) to get
\begin{align*}
\mathcal{N}_r(z)&=\E[(1+z+2\sqrt{z}X)^{r-3}(1+z+2\sqrt{z}X)^2]\\
&=(1+z)^2\mathcal{N}_{r-2}(z)+4(1+z)\sqrt{z}\E[(1+z+2\sqrt{z}X)^{r-3}X]+4z\E[(1+z+2\sqrt{z}X)^{r-3}X^2],
\end{align*}
which, together with (\ref{io}) and (\ref{ioa}), yields
\begin{align*}
\mathcal{N}_r(z)&=(1+z)^2\mathcal{N}_{r-2}(z)+2(1+z)\left(\mathcal{N}_{r-1}(z)-(1+z)\mathcal{N}_{r-2}(z)\right)\\
&-\frac{3}{r-2}\left(\mathcal{N}_r(z)-(1+z)\mathcal{N}_{r-1}(z)\right)+4z\mathcal{N}_{r-2}(z).
\end{align*}
So it holds that
$$\frac{r+1}{r-2}\mathcal{N}_r(z)=\frac{2r-1}{r-2}(1+z)\mathcal{N}_{r-1}(z)-(1-z)^2\mathcal{N}_{r-2}(z)$$
and (\ref{ikldd}) follows.

\noindent {\it Proof of (\ref{ikld}).} It follows from the representation (\ref{nana}) that
$$\mathcal{N}'_r(z)=(r-1)\E\left[(1+z+2\sqrt{z}X)^{r-2}\left(1+\frac{X}{\sqrt{z}}\right)\right],$$
which implies that
\begin{equation}\label{ioo}
\mathcal{N}'_r(z)-(r-1)\mathcal{N}_{r-1}(z)=\frac{r-1}{\sqrt{z}}\E\left[(1+z+2\sqrt{z}X)^{r-2}X\right].
\end{equation}
Combining (\ref{io}) and (\ref{ioo}) yields
$$\mathcal{N}'_r(z)-(r-1)\mathcal{N}_{r-1}(z)=\frac{r-1}{2z}\left(\mathcal{N}_r(z)-(1+z)\mathcal{N}_{r-1}(z)\right).$$
So (\ref{ikld}) follows.

\noindent {\it Proof of (\ref{iklda}).} We have
\begin{align*}
N''_r(z)&=(r-1)(r-2)\E\left[(1+z+2\sqrt{z}X)^{r-3}\left(1+\frac{X}{\sqrt{z}}\right)^2\right]\\
&-\frac{r-1}{2}\E\left[(1+z+2\sqrt{z}X)^{r-2}\frac{X}{z\sqrt{z}}\right]\\
&=(r-1)(r-2)\E\left[(1+z+2\sqrt{z}X)^{r-3}\left(1+\frac{2X}{\sqrt{z}}+\frac{X^2}{z}\right)\right]\\
&-\frac{r-1}{2}\E\left[(1+z+2\sqrt{z}X)^{r-2}\frac{X}{z\sqrt{z}}\right].
\end{align*}
Hence, recalling the relations (\ref{io}) and (\ref{ioa}), we obtain
\begin{align*}
\mathcal{N}''_r(z)&=(r-1)(r-2)\bigg(\mathcal{N}_{r-2}(z)+\frac{1}{z}\left(\mathcal{N}_{r-1}(z)-(1+z)\mathcal{N}_{r-2}(z)\right)\\
&-\frac{3}{4(r-2)z^2}\left(\mathcal{N}_r(z)-(1+z)\mathcal{N}_{r-1}(z)\right)+\frac{1}{z}\mathcal{N}_{r-2}(z)\bigg)\\
&-\frac{r-1}{4z^2}\left(\mathcal{N}_r(z)-(1+z)\mathcal{N}_{r-1}(z)\right).
\end{align*}
Rearranging the terms we get
$$\mathcal{N}''_r(z)=-\frac{r-1}{z^2}\mathcal{N}_r(z)+\left(\frac{(r-1)^2}{z}+\frac{r-1}{z^2}\right)\mathcal{N}_{r-1}(z).$$
Thus the proof of (\ref{iklda}) is complete.


\begin{thebibliography}{99}

\bibitem{Amdeberhan2013} T. Amdeberhan, V. H. Moll, C. Vignat, A probabilistic interpretation of a sequence related to Narayana polynomials. {\it Online J. Anal. Comb.} No. 8 (2013), 25 pp.

\bibitem{Chen2010} W. Y. C. Chen, L. X. W. Wang, A. L. B. Yang, Schur positivity and the $q$-log-convexity of the Narayana polynomials. {\it J. Algebraic Combin.} 32 (2010), no. 3, 303--338.

\bibitem{Lassalle2012} M. Lassalle, Two integer sequences related to Catalan numbers. {\it J. Combin. Theory Ser. A} 119 (2012), no. 4, 923--935.

\bibitem{Mansour2009} T. Mansour, Y. Sun, Identities involving Narayana polynomials and Catalan numbers. {\it Discrete Math.} 309 (2009), no. 12, 4079--4088.
\end{thebibliography}
\end{document}